\documentclass[12pt, reqno]{amsart}
\usepackage{amssymb, amstext, amscd, amsmath}
\usepackage{color}

\makeatletter
\def\@cite#1#2{{\m@th\upshape\bfseries%
[{#1\if@tempswa{\m@th\upshape\mdseries, #2}\fi}]}} \makeatother
\theoremstyle{plain}
\newtheorem{theorem}{Theorem}[section]
\newtheorem{thm}[theorem]{Theorem}
\newtheorem{cor}[theorem]{Corollary}
\newtheorem{prop}[theorem]{Proposition}
\newtheorem{lem}[theorem]{Lemma}
\theoremstyle{definition}
\newtheorem{rem}[theorem]{Remark}

\newtheorem{defn}[theorem]{Definition}

\newtheorem{eg}[theorem]{Example}

\newtheorem{ques}[theorem]{Question}
\newcommand{\bA}{{\mathbb{A}}}
\newcommand{\bB}{{\mathbb{B}}}
\newcommand{\bC}{{\mathbb{C}}}
\newcommand{\bD}{{\mathbb{D}}}

\newcommand{\bF}{{\mathbb{F}}}

\newcommand{\bN}{{\mathbb{N}}}

\newcommand{\bT}{{\mathbb{T}}}

  \newcommand{\A}{{\mathcal{A}}}
  \newcommand{\B}{{\mathcal{B}}}
  \newcommand{\C}{{\mathcal{C}}}

  \newcommand{\F}{{\mathcal{F}}}
  
\renewcommand{\H}{{\mathcal{H}}}
  
  \newcommand{\J}{{\mathcal{J}}}
  \newcommand{\K}{{\mathcal{K}}}
\renewcommand{\L}{{\mathcal{L}}}
\newcommand{\M}{{\mathcal{M}}}
  \newcommand{\N}{{\mathcal{N}}}

\renewcommand{\S}{{\mathcal{S}}}
  \newcommand{\T}{{\mathcal{T}}}
  
  \newcommand{\V}{{\mathcal{V}}}

\newcommand{\upchi}{{\raise.35ex\hbox{\ensuremath{\chi}}}}

\newcommand{\alg}{\operatorname{Alg}}

\newcommand{\id}{{\operatorname{id}}}

\newcommand{\spn}{\operatorname{span}}

\newcommand{\Tr}{\operatorname{Tr}}

\newcommand{\Fd}{\mathbb{F}_d^*}

\newcommand{\mt}{\varnothing}

\newcommand{\ol}{\overline}

\newcommand{\wot}{\textsc{wot}}

\newcommand{\otl}{\, \ol\otimes\, }

\begin{document}
\title[Hopf structure of dual algebras]{The Hopf structure of some \\ dual operator algebras}
\author[M. Kennedy]{Matthew Kennedy}
\address{
Matthew Kennedy, School of Mathematics and Statistics, Carleton University, 1125 Colonel
By Drive, Ottawa, Ontario K1S 5B6, Canada}
\email{mkennedy@math.carleton.ca}
\author[D. Yang]{Dilian Yang}
\address{
Dilian Yang, Department of Mathematics $\&$ Statistics, University of Windsor, Windsor, ON
N9B 3P4, CANADA}
\email{dyang@uwindsor.ca}

\begin{abstract}
We study the Hopf structure of a class of dual operator algebras corresponding to certain semigroups. This class of algebras arises in dilation theory, and includes the noncommutative analytic Toeplitz algebra and the multiplier algebra of the Drury-Arveson space, which correspond to the free semigroup and the free commutative semigroup respectively.  The preduals of the algebras in this class naturally form Hopf (convolution) algebras. The original algebras and their preduals form (non-self-adjoint) dual Hopf algebras in the sense of Effros and Ruan.
We study these algebras from this perspective, and obtain a number of results about their structure.
\end{abstract}

\thanks{Both authors partially supported by NSERC}

\date{}
\maketitle

In this paper, we study a class of dual operator algebras with the property that the algebras and their preduals form compatible Hopf algebras in the sense of Effros and Ruan \cite{ER1} (see also Rieffel's paper \cite{Rie}). This class of algebras arises in multivariate operator theory, and includes the noncommutative analytic Toeplitz algebra and the multiplier algebra of the Drury-Arveson space.

A dual operator algebra $\A$ is a unital algebra of bounded linear operators on a Hilbert space $\H$ that is closed in the weak* topology. Such an algebra arises in a canonical way as the dual of a quotient of the algebra of trace class operators on $\H$. Namely,
\[
\A = (S_1(\H)/\A_\perp)^*,
\]
where $S_1(\H)$ denotes the algebra of trace class operators on $\H$, and $\A_\perp$ denotes the preannihilator of  
$\A$, that is
\[
\A_\perp = \{T \in S_1(\H) \mid \Tr(TA) = 0 \text{ for all } A \in \A\}.
\]
We will write $\A_* = S_1(\H)/\A_\perp$, and we will refer to this space as the (standard) predual of $\A$.

We also mention Le Merdy's work in \cite{Lem99} which provides an abstract characterization of dual operator algebras, and can be seen as the weak* analogue of the Blecher-Ruan-Sinclair theorem. Since the algebras we study have natural concrete representations, we do not require these abstract results.

Analyzing the structure of the predual of a dual operator algebra can reveal a great deal about the algebra itself (see for example \cite{Ber10, BFP, BL05, Ken11, Ken12}). However, the above presentation of the predual is often not very revealing. It is of interest, therefore, to identify when the predual is endowed with some kind of additional structure.

The noncommutative analytic Toeplitz algebra $\L_d$ is the unital weakly closed operator algebra generated by the left regular representation of the free semigroup $\Fd$ on the full Fock space $\ell^2(\Fd)$. The algebra $\L_d$ was first studied by Popescu in \cite{Pop95}, but it has recently received a great deal of attention due to its prominent role in multivariate operator theory. For more information on this work, we direct the reader to \cite{DavYang, Ken11, Ken12, Pop89, Pop95, Pop96}, and the references contained therein.

In \cite{Yang12}, the second author studied the Hopf structure of $\L_d$ and its predual $(\L_d)_*$. It was shown there that $(\L_d)_*$ is a commutative Hopf algebra (and in particular, a commutative Banach algebra) with spectrum $\Fd$. The algebra $\L_d$ also forms a Hopf algebra, and the comultiplication on $\L_d$ arises as the dual of the multiplication on $(\L_d)_*$. Moreover, there is one-to-one correspondence between the set of corepresentations of $\L_d$, and the completely bounded representations of $(\L_d)_*$.

The algebras of interest in this paper arise as quotients of $\L_d$. These quotients provide a large class of interesting algebras, which includes the multiplier algebra of every irreducible complete Nevanlinna-Pick space, and in particular, the Drury-Arveson space ( see \cite{AM00, AM02, DavPit98}). More recently, these quotient algebras were studied in \cite{KenYan13}. There, a non-self-adjoint analogue of the classical Lebesgue decomposition was obtained and used to prove that these algebras have strongly unique preduals.

We view the noncommutative analytic Toeplitz algebra $\L_d$ as the analogue, for the free semigroup $\Fd$, of a group von Neumann algebra. From this perspective, the class of algebras we are interested in correspond to quotient semigroups of $\mathbb{F}_d^*$ arising from semigroup congruences (for technical reasons we will need to adjoin a zero element to $\Fd$) . For example, we will show that the multiplier algebra of the Drury-Arveson space in $d$ variables corresponds to the free commutative semigroup $\mathbb{N}^d$.

The paper is organized as follows. In Section \ref{S:pre}, we provide the background material we will require, and in particular introduce Effros and Ruan's notion of a Hopf algebra. In Section \ref{S:ccap}, we show that certain quotients of $\L_d$ possess approximation properties that will play an important role in our later results.

In Section \ref{S:hhideal}, we study homogeneous Hopf ideals of $\L_d$, which are generated by homogeneous polynomials in the generators of $\L_d$. We prove a correspondence between these ideals and quotient semigroups of $\mathbb{F}_d^*$ arising from certain semigroup congruences. Let $\J$ be a homogeneous Hopf ideal. We give a functorial construction of a weighted left regular representation of the corresponding semigroup, and show this representation extends to a completely isometric representation of $\L_d/\J$. Moreover, we show this representation is unitarily equivalent to a compression of the algebra $\L_d$ to a coinvariant subspace. Being able to identify $\L_d/\J$ with its image in this representation will be very useful in our work.

In Section \ref{S:LJ}, we prove that for an arbitrary Hopf ideal $\J$, the quotient algebra $\L_d/\J$ is a Hopf algebra. If, furthermore, $\J$ is homogeneous, then we prove that the 
predual of $\L_d/\J$ is a commutative Hopf (convolution) algebra. In this case, the predual of $\L_d/\J$ is a commutative Banach algebra. In Section \ref{S:spec},we prove the spectrum of such an algebra consists of the nonzero elements of the corresponding semigroup. We also investigate the case when $\J$ is non-homogeneous. 

In Section \ref{S:aut}, for a homogeneous Hopf ideal $\J$ with corresponding semigroup $S$, we completely characterize the Hopf algebra automorphisms of $\L_d/\J$ and its predual. We prove that, roughly speaking, they are completely determined by the semigroup automorphisms of $S$. Finally, in Section \ref{S:mul}, we briefly discuss the Schur multipliers of the predual of $\L_d$.

\section{Preliminaries}
\label{S:pre}

In this section, we provide some of the background which will be needed later, and in addition fix our notation. 
The main references for this material are \cite{AP1, Dav2, DavPit, DavPit98, DavPit99, DavRamSha2, Yang12}.

\subsection{Free semigroup algebras}

Let $\Fd$ be the unital free semigroup on $d$ generators $\{1,...,d\}$, and let $\ell^2(\Fd)$ 
be the Fock space with orthonormal basis $\{\xi_w: w\in \Fd\}$. The free semigroup algebra
$\L_d$ is the unital $\wot$-closed algebra generated by the left regular representation of $\Fd$ on $\ell^2(\Fd)$. In other words,
$\L_d$ is generated by the operators $L_u(\xi_w)=\xi_{uw}$ for all $u,w\in \Fd$. 

As shown in \cite{DavPit99}, every element $A$ in $\L_d$ is uniquely determined by its Fourier series
\[
A\sim\sum_{w\in\Fd}a_wL_w,
\]
in the sense that 
\[
A\xi_w=\sum_{w\in\Fd}a_w\xi_w\quad  (w\in\Fd),
\]
where the $w$-th Fourier coefficient of $A$ is given by $a_w=(A\xi_\mt,\xi_w)$.
It is often useful heuristically to work directly with the Fourier series of $A$.

For $k\ge 1$, the $k$-th Ces\`aro map of the Fourier series of $A$ is defined by 
\[
\Sigma_k(A)=\sum_{|w|<k}\Big(1-\frac{|w|}{k}\Big)a_wL_w,
\]
where $|w|$ is the length of the word $w$ in $\Fd$. 
The sequence $\{\Sigma_k(A)\}_k$  converges to $A$ in the strong operator topology (see for example \cite{DavPit99}).

Let $\J$ be a $\wot$-closed ideal of $\L_d$. We are interested in the quotient algebra $\L_d/\J$. To facilitate the study of this algebra, we will realize it as the compression of the algebra $\L_d$ to a coinvariant subspace. The orthogonal complement of the range of $\J$ is the space $\N_\J = (\J\xi_\varnothing)^\perp$. Note that $\N_\J$ is a coinvariant subspace for $\J$. Let $P_{\N_\J}$ be the orthogonal projection onto  $\N_\J$. The compression algebra $\L_\J$ is defined as 
\[\L_\J=P_{\N_\J}\L_d|_{\N_\J}.
\] 
It turns out that 
$\L_d/{\J}$ and $P_{\N_\J}\L_d|_{\N_\J}$ are weak*-homeomorphic and completely isometrically isomorphic \cite{AP1,DavPit98}. 

The algebra $\L_\J$ enjoys many nice properties. For example, it is known that the weak* and $\wot$ topologies on $\L_\J$ coincide (\cite[Proposition 4.3]{AP1} and \cite[Theorem 2.1]{DavPit98}).  We will make frequent use of this fact. More recently, we studied these quotient algebras in \cite{KenYan13}, and showed that they have strongly unique preduals.

Let $\bB_d$ denote the open unit ball of $\bC^d$, and for $\lambda\in\bB_d$ let
\[
\nu_\lambda=(1-\|\lambda\|^2)^\frac{1}{2}\sum_{w\in\Fd}\ol{w(\lambda)}\xi_w. 
\]
Then every $\wot$-continuous multiplicative linear functional on $\L_d$ 
can be represented by a rank one linear functional
of the form $\varphi_\lambda=[\nu_\lambda\nu_\lambda^*]$,  where
\[
\varphi_\lambda(A)=(A\nu_\lambda, \nu_\lambda)\quad (A\in \L_d). 
\]
This fact is proved in \cite{DavPit}.

\subsection{The commutative case}


Let $\C$ be the $\wot$-closed commutator ideal of $\L_d$ generated by $\{L_iL_j-L_jL_i:i,j=1,\ldots, d\}$.
Then 
\[
\C=\cap_{\lambda\in\bB_d}\ker \varphi_\lambda,
\]
and $\N_\C$ is the symmetric Fock subspace of $\ell^2(\Fd)$, i.e. the \textit{Drury-Arveson space}, which we denote by $\H_d^2$ (see for example \cite{Arv98}). 
It is known that $\H_d^2$ is a reproducing kernel space of functions on $\bB_d$
with kernel functions 
\[
k_\lambda=(1-\|\lambda\|^2)^{-\frac{1}{2}}\, \nu_\lambda, \quad \lambda\in\bB_d.
\] 

The multiplier algebra of $\H_d^2$ is the algebra $\M_d$. Thus, the coordinate functions $Z_1,\ldots,Z_d$ on $\H_d^2$ are given by the 
compressions $Z_i=P_{\H_d^2}L_i|_{\H_d^2}$, $i=1,...,d$. That is,
\[
Z_i(f)(z) = (z_if)(z) \quad \text{for all}\quad  f\in\H_d^2.
\]
For $k = (k_1,...,k_d) \in \mathbb{N}^d$,
we use standard multi-index notation and write
$Z^k = Z_1^{k_1} ... Z_d^{k_d}$ and $z^k = z_1^{k_1} ... z_d^{k_d}$. The $z^k$'s form an orthogonal (but not orthonormal unless $d = 1$) basis for $\H_d^2$, since
\[
(z^k, z^\ell)=\frac{k!}{|k|!}  \delta_{k\ell}\quad k,\ell\in\bN^d.
\]
For more information, see, for example, \cite{Arv98, DavPit, DavPit98}.

More generally, let $V$ be an analytic variety of $\bB_d$, i.e., the common set of zeros of a family of functions in $\H_d^2$. Let 
\[
\J_V=\{f\in\M_d: f(\lambda)=0 \  \forall \lambda\in V\},
\]
and let
\[
\F_V=\ol\spn\{k_\lambda:\lambda\in V\}.
\]
Then $\J_V$ is a $\wot$-closed ideal in $\M_d$, and $\F_V$ is a reproducing kernel Hilbert space of functions on $V$. 
Let $\M_V$ be the multiplier algebra of $\F_V$. It is proved in \cite{DavRamSha2} that 
 \[
 \M_V=\{f|_V:f\in\M_d\},
 \]
 and that, moreover, the mapping 
 \[
 \phi:\M_d\to \M_V, \quad \phi(f)=f|_V
 \] 
is a weak* homeomorphism and a completely isometric isomorphism between $\M_d/\J_V$ and $\M_V$. Furthermore,  $\M_V=P_{\F_V}\M_d|_{\F_V}$.

\subsection{Hopf dual algebras}
A dual operator algebra $\A$ is said to be a \textit{Hopf dual algebra} if there is a unital
weak*-continuous completely contractive
homomorphism $\Delta_\A:\A\to \A\ol\otimes \A$
that is coassociative, meaning
\[
(\id\otimes \Delta_\A)\Delta_\A=(\Delta_\A\otimes \id)\Delta_\A.
\]
By $\A\ol\otimes\A$, we mean the normal tensor product of $\A$ with itself. 
The homomorphism $\Delta_\A$ is called a \textit{comultiplication} or a \textit{coproduct} on $\A$.
We will use write $(\A, \Delta_\A)$ to denote the pairing of $\A$ with a comultiplication $\Delta_\A$. Occasionally, we will drop the subscript and write $(\A, \Delta_\A)$, or even just $\A$ if the context is clear.

Let $(\A, \Delta_\A)$ be a Hopf dual algebra. A weak*-closed subspace $\J$ of $\A$ is called a \textit{coideal} if 
\[
\Delta_\A(\J)\subseteq \ol{\J\otimes\A+\A\otimes\J}^{w^*}.
\]
If $\J$ is both an ideal and coideal of $(\A, \Delta_\A)$, $\J$ is said to be a \textit{Hopf ideal}.

We follow the lead of Effros and Ruan in \cite{ER1} (see also \cite{Rie}), and use the following (analysis-centric) notion of a Hopf algebra.
A \textit{Hopf algebra} $(\A,  m, \Delta)$ consists of a linear space $\A$ with norms or matrix norms, an associative bilinear multiplication $m: \A\times \A\to \A$,
and a comultiplication $\Delta_\A: \A\to \A\widetilde{\otimes} \A$. Here, $\widetilde\otimes$ is a suitable tensor product. The maps $m$ and $\Delta$ are both assumed
to be bounded in an appropriate sense.

In \cite{Yang12}, it was shown that $\L_d$ is a Hopf dual algebra with comultiplication $\Delta:\L_d \to \L_d \overline{\otimes} \L_d$ defined by
\[
\Delta(L_w) = L_w\otimes L_w,\quad w\in\Fd.
\]
Furthermore, it was shown that the predual $({\L_d})_*$ of $\L_d$ is also a Hopf algebra which, in analogy with the classical notion of a Hopf algebra, was called a \textit{Hopf convolution algebra}. 
In particular,  the algebra structure on $({\L_d})_*$ is induced by the comultiplication of $\L_d$. That is, for all $\varphi, \psi\in ({\L_d})_*$, 
\[
\varphi \star \psi (A) = (\varphi\otimes \psi) \circ \Delta(A),\quad A\in\L_d.
\]


\subsection{Notation} 

Since $d$ is fixed throughout the paper, for simplicity we often write $\L$, $\M$, $\L_*$ and $\M_*$ for 
$\L_d$, $\M_d$, $({\L_d})_*$ and $({\M_d})_*$, respectively. To avoid confusion, we use the notation $\star$ to denote 
the multiplication on the predual ${\L}_*$ (and similarly on the predual of the algebras under investigation). We reserve the use of the asterisk $*$ to denote the dual or predual of a space, and the adjoint of an an operator or map).

For the rest of the paper, by an ideal of $\L$ or $\M$ we will always mean a $\wot$-closed (and hence weak*-closed from above) two-sided ideal. 


\section{The completely contractive approximate property}
\label{S:ccap} 

In this section, we show that for every homogeneous ideal $\J$ of $\L$, the compression
algebra $\L_\J$ has the completely contractive approximation property (CCAP), and hence has property $S_\sigma$ from \cite{Kraus1, Kraus2}. These facts will play an important role in our results.

It was shown in \cite{Bar} that multiplier algebras of circular reproducing kernel Hilbert spaces have the CCAP. In our context, an analytic variety $V$ is circular if and only if $e^{it}V \subseteq V$ for every $t \in \bT$. If $\J$ is a homogeneous ideal of $\L$ containing the commutator ideal $\C$, then the algebra $\L_\J$ is the multiplier algebra of a circular reproducing kernel Hilbert space.

\begin{defn}
Let $\S$ be a weak*-closed subspace of $B(\H)$. Then $\S$ is
said to have the \textit{completely contractive approximation property} (CCAP) if there is a net of completely
contractive weak*-continuous finite rank maps 
$\Phi_i: \S\to\S$ such that
\[
w^*\!\!-\!\!\lim_i\Phi_i(A)=A, \quad A\in\S.
\]
The space $S$ is said to have the \textit{completely bounded approximation property} (CBAP) if the maps $\Phi_i$ are instead only completely bounded.
\end{defn}

\begin{defn}
Let $\S$ be a weak*-closed subspace of $B(\H)$. Then $\S$ is said to have \textit{property $S_\sigma$} if
\[
\S\otimes_\F \T=\S\ol\otimes \T,
\]
whenever $\T$ is a weak*-closed subspace of $B(\K)$ for some Hilbert space $\K$. 
Here, $\S\otimes_\F\T$ stand for the Fubini tensor product of $\S$ and $\T$ (see for example \cite{Kraus1}).
\end{defn}

It is known that the CBAP (and hence the CCAP) implies property $S_\sigma$, but the converse is not true \cite{Kraus1, Kraus2}. 
The space $\S$ has property $S_\sigma$ if and only if it has the weak* operator approximation property, which happens if and only if the predual  
$\S_*$ of $\S$ has the operator approximation property \cite{KL}.

\begin{thm}\rm{(Kraus)}
\label{T:Kraus}
Let $\S$ be a weak*-closed subspace of $B(\H)$. If $\S$ has the CBAP, then $\S$ has property $S_\sigma$.
\end{thm}

\begin{thm}
\label{T:ccap}
Let $\J$ be a homogeneous ideal of $\L$. Then
the compression algebra $\L_\J$ has the CCAP, and so has property $S_\sigma$.
\end{thm}

\begin{proof} 
Let $\V =\N_\J$, so that
$\L_\J =P_\V \L|_{\V}$. Since $\J$ is homogeneous, we can write
\[
\V=\V_0\oplus \V_1\oplus \cdots 
\]
where for $k\ge 0$,
\[
\V_k\subseteq \F_k:=\spn\{\xi_w: w\in \Fd, |w|=k\}.
\]
Let $\Phi_k$ denote the completely contractive map from $\L_\J$ to $B(\V)$ given by
\[
\Phi_k(A)=\sum_{j\ge 0}P_{\V_{j+k}}A P_{\V_j}, \quad A\in \L_\J,
\]
and define the Ces\`aro map from $\L_\J$ to $B(\V)$ by
\[
\widehat\Sigma_k(A)=\sum_{j=0}^k \Big(1-\frac{j}{k} \Big)\Phi_j(A), \quad A\in \L_\J. 
\]
Let $\pi$ be the natural compression map from $\L$ to $\L_\J$. Then $\pi$ is completely contractive and 
weak*-weak* continuous. Since the projection $P_\V$ leaves each $\V_k$ invariant, we have 
\begin{align}
\label{E:cesaro}
\widehat\Sigma_k(\pi(A))=\pi(\Sigma_k(A)), \quad A\in\L,
\end{align}
where $\Sigma_k$ is the Ces\`aro map on $\L$ from the preliminaries. Hence each $\widehat\Sigma_k$ is a unital,
finite rank, completely contractive map from $\L_\J$ to $\L_\J$.
It follows from the corresponding properties of the map $\Sigma_k$ (see for example \cite{DavPit99}) that 
\begin{align}
\label{E:wlim}
w^*\!\!-\!\!\lim_k \widehat\Sigma_k(A)=A, \quad A\in\L_\J.
\end{align}
Therefore, $\L_\J$ has the CCAP. 
\end{proof}


Two immediate consequences of the above theorem are recorded below for later reference.
A proof of the first result can be seen in \cite{Ruan}.

\begin{cor}
\label{C:Fubini}
Let $\J$ be a homogeneous ideal of $\L$.  Then we have the following completely isometric
isomorphism
\[
(\L_\J\ol\otimes\T)_* \cong (\L_\J)_*\widehat\otimes \T_*,
\]
whenever $\T$ is a weak*-closed subspace of $B(\K)$ for any Hilbert space $\K$. 
Here $(\L_\J)_*\widehat\otimes \T_*$ denotes the operator projective tensor product of  $(\L_\J)_*$ and $\T_*$. 
\end{cor}

\begin{cor}
$\L$ and $\M$ both have the CCAP and so property $S_\sigma$. 
\end{cor}

\begin{proof}
Let $\J=\{0\}$ and $\C$ in Theorem \ref{T:ccap}. 
\end{proof}

\begin{rem}
Let $\J$ be a homogeneous ideal of $\L$, and let $S_i^\J=P_{\N_\J}L_i|_{\N_\J}$  for $i=1,\ldots,d$. 
From the proof of Theorem \ref{T:ccap}, one obtains that every $A$ in $\L_\J$ has a Fourier series of the form
\[
A\sim \sum_{w\in\Fd} a_wS_w^\J,
\]
and that
\[
A=w^*\!\!-\!\!\lim_k \sum_{|w|<k}\Big(1-\frac{|w|}{k}\Big) a_wS_w^\J. 
\]
These facts  will be useful later. 
\end{rem}

It would be nice to know whether $\L_\J$ has one or both of the CBAP and property $S_\sigma$ when $\J$ is an arbitrary (i.e. not necessarily homogeneous) ideal.
(See also Question \ref{Q:conv} in Section \ref{S:LJ} below.)

\section{Homogeneous Hopf ideals}
\label{S:hhideal}

Homogeneous Hopf ideals provide a nice class of ideals in $\L$. They are determined by
homogeneous semigroup congruences of $\bF_d^{*0}$, the semigroup obtained by adjoining a zero element $0$
to $\Fd$. Let $\J$ be a homogeneous ideal of $\L$, and let $S$ denote the  semigroup quotient of $\bF_d^{*0}$ determined 
by the corresponding semigroup congruence. We will prove that the algebra $\L_\J$
is unitarily equivalent to a semigroup algebra $\L[S]$ generated by a weighted left regular representation of $S$.
This unitary equivalence will facilitate our study of the algebra $\L_\J$.

\subsection{The structure of homogeneous coideals}
\label{subS:coideal}

Let $D_n(\bC)$ denote the abelian subalgebra of  $M_n(\bC)$ consisting of all diagonal matrices.  
The following lemma is an easy consequence of the fact that $D_n(\bC)$ is an abelian von Neumann algebra. 

\begin{lem}
\label{L:basis}
For $n\geq1$, let $\mathcal{A}$ be a nonzero subalgebra of $D_n(\bC)$. 
Then there is a basis  for $\mathcal{A}$ consisting of pairwise orthogonal projections.
\end{lem}

For $w\in\Fd$, let $\varphi_w$ be the ``coefficient functional" on $\L$, given by
\[
\varphi_{w}(A)=[\xi_\mt\xi_w^*](A)=(A\xi_\mt, \xi_w), \quad A\in \L,
\]
which ``peels'' off the $w$-th Fourier coefficient of $A$. Note that
\[
\varphi_{w}(L_{w'})=\delta_{w,w'},\quad w,w'\in\Fd.
\]
These functionals are evidently weak*-continuous. It is easy
to see that they weakly span the predual of $\L_*$.

\begin{prop}
\label{P:coideal}
Let $\J\subset \L$. Then $\mathcal{J}$ is a homogeneous coideal of $\L$ if and only if 
there are disjoint subsets $W,W_1,W_2,\ldots$ of $\Fd$ with $W_n\subset \{w\in\Fd: |w|=n\}$
such that 
\[
\mathcal{J}
=\ol{\spn}^{w^*} \big(\{L_{u}-L_{v} \mid u,v \in W_n, \, n \geq 1 \} \cup \{ L_w \mid w \in W\}\big).
\]
\end{prop}

\begin{proof}
($\Leftarrow$) The proof of this implication is straightforward. Clearly $\J$ is homogeneous, and $\J$ is a coideal since  $\Delta(L_w)=L_w\otimes L_w\in \J\otimes \J$ and  
\begin{align*}
\Delta(L_u-L_v)
&=L_u\otimes L_u-L_v\otimes L_v\\
&=(L_u-L_v)\otimes L_u +L_v\otimes (L_u-L_v)\\
&\in \J\otimes \L+ \L\otimes \J
\end{align*}
for all $u,v,w\in\Fd$. 

($\Rightarrow$)
Note that $\mathcal{L}$ is weak*-spanned by its homogeneous elements.
For $n\geq1$, let $\mathcal{L}^{n}$ denote the set of homogeneous
elements in $\mathcal{L}$ of degree $n$, i.e.
\[
\mathcal{L}^n=\mathrm{span}\left\{ \sum_{|w|=n}a_{w}L_{w}\right\}. 
\]
Similarly, the Hopf convolution algebra $\mathcal{L}_{*}$ is also weakly spanned by its homogeneous
elements. For $n\geq1$, let $\mathcal{L}_{*}^{n}$ denote the
set of homogeneous elements in $\mathcal{L}_{*}$ of degree $n$, i.e.
\[
\mathcal{L}_{*}^{n}=\mathrm{span}\left\{ \sum_{|w|=n}a_{w}\varphi_{w}\right\}.
\]
Notice that each $\mathcal{L}_{*}^{n}$ is actually a subalgebra of
$\mathcal{L}_{*}$ which is naturally isomorphic to the diagonal matrix algebra $D_{d^n}(\bC)$.

Since $\mathcal{J}$ is a coideal of $\L$, one can easily verify that its preannihilator $\J_\perp$ is a subalgebra of $\L_{*}$.
Indeed, for $\varphi,\psi\in\J_\perp$ and $A\in \J$, we have 
\[
\varphi\star\psi(A)=(\varphi\otimes\psi)(\Delta_\L(A))=0,
\] 
by the properties of a coideal.
Since $\J$ is homogeneous, it is spanned by its homogeneous elements.
For $n \geq 1 $,  let $\mathcal{J}_{n}=\mathcal{J}\cap\mathcal{L}^{n}$ denote the set of homogeneous elements in $\mathcal{J}$ of degree
$n$, and set $\mathcal{J}_{\perp}^{n}=\mathcal{J}_{\perp}\cap\mathcal{L}_{*}^{n}$.
This is easily seen to be a subalgebra of $\mathcal{L}_{*}^n$, and hence by Lemma \ref{L:basis}, 
we can find a basis $\psi_{1},\ldots,\psi_{m}$ for $\mathcal{J}_{\perp}^{n}$ consisting of pairwise orthogonal projections.
Hence there are disjoint subsets $W_{1},\ldots,W_{m}$ of $\{w\in\Fd: |w|=n$\} such that 
\[
\psi_{i}=\sum_{w\in W_{i}}\varphi_{w}, \quad 1\le i\leq m.
\]
Set $W = \{w \in \Fd \mid |w| = n, \, w \notin \cup_{i = 1}^m W_i \}$. Then it follows that 
\begin{align*}
\mathcal{J}_{n}
&=\spn\left\{\sum_{u\in\Fd} a_uL_u \in \L^n: \sum_{w\in W_i} a_w =0, \, i=1,...,m\right\}\\
&=\spn \big(\{L_{u}-L_{v} \mid u,v \in W_i, \, 1 \leq i \leq m \} \cup \{ L_w \mid w \in W\}\big),                  
\end{align*}
which gives the desired form of $\J$.
\end{proof}

\subsection{The structure of homogeneous Hopf ideals}

For a semigroup $G$, a \textit{semigroup congruence} on $G$ is an equivalence relation which is compatible with the 
semigroup operation.  As above, we write $G^0$ to denote the semigroup obtained by adjoining a zero element $0$ to $G$. 
We will consider $\bF_d^{*0}$, but notice that the length map is only defined on $\Fd$. We use the convention that $L_0:=0$, the zero operator in $\L$.
So $L_0$ is always in every ideal of $\L$. Recall that $\varnothing$ denotes the identity element in $\Fd$. We use the standard convention that $L_\varnothing = I$.

\begin{defn}
A semigroup congruence $\sim$ on $\bF_d^{*0}$ is said to be \textit{homogeneous} if whenever $u\sim v$ with $u\not\sim 0$ (and/or $v\not\sim 0$) 
implies $|u|=|v|$. 
\end{defn}


\begin{thm}
\label{T:Hopfideal} 
Let $\J\subset\L$. Then the following statements are equivalent:
\begin{itemize}
\item[(i)] The set $\mathcal{J}$ is a homogeneous Hopf ideal of $\mathcal{L}$.
\item[(ii)] There is a homogeneous semigroup congruence $\sim$ on $\bF_d^{*0}$ such that $\J$ is of the form
\[
\mathcal{J}=\ol{\spn}^{w^*}\{L_{u}-L_{v} \mid u,v\in \bF_d^{*0} \text{ with } u\sim v\}.
\]
\end{itemize}
\end{thm}

\begin{proof}
(i) $\Rightarrow$ (ii)
Suppose that $\J$ is a homogeneous Hopf ideal of $\L$. Let us identify $\L_\J$ with $\L/\J$. Define $h: \mathbb{F}_d^{*0}\to \L/\J$ by 
\begin{align}
\label{E:h}
h(w)=L_w+\J, \quad w\in \mathbb{F}_d^{*0}.
\end{align}
Clearly $h$ is a semigroup homomorphism. Thus its kernel $\ker h$ gives rise to a semigroup congruence $\sim$.
Namely, $u\sim v$ if and only if  $h(u)=h(v)$ for $u, v\in \mathbb{F}_d^{*0}$. In other words, $u\sim v$ if and only if $L_u-L_v\in \J$. 
This gives $\ol\spn^{w^*}\{L_u-L_v \mid  u\sim v\}\subset \J$, and the reverse inclusion follows directly from Proposition \ref{P:coideal}. 

\medskip
(ii) $\Rightarrow$ (i)
By Proposition \ref{P:coideal}, $\J$ is a homogeneous coideal. So it remains  to prove that $\J$ is an ideal.
But this is straightforward: Suppose $u \sim v$, so that $L_u-L_v$ is an element of $\J$. For every word $w \in \mathbb{F}_d^{*0}$, 
$wu \sim wv$ and $uw \sim vw$ since $\sim$ is a semigroup congruence. Hence $L_w(L_u-L_v)\in\J$ and $(L_u-L_v)L_w\in \J$, which proves $\J$ is an ideal. 
\end{proof}

\subsection{Semigroup algebras associated to homogeneous Hopf ideals}
\label{subS:L[S]}

Let $\J$ be a homogeneous Hopf ideal of $\L$, and $\sim$ be the corresponding semigroup congruence on $\bF_d^{*0}$ as in Theorem \ref{T:Hopfideal}.

We will use $S$ to denote the quotient  semigroup $\bF_d^{*0}/\sim$, and we will use $S_0$ to denote the
subset of all nonzero elements of $S$, i.e. $S_0 =S\setminus \{\dot{0}\}$. (Here $\dot{0}$ denotes the image of $0$ in $S$.)

Clearly, $S$ is unital with zero element $\dot{0}$.
While $S_0$ is not a semigroup in general, it is not hard to see that $S_0$ is a (unital) semigroup if and only if $S$ is non-trivial and has no nonzero zero-divisors. 

\begin{eg}
\label{Eg:Hopfideal}
We now give some examples illustrating the correspondence between homogeneous ideals of $\L$ and semigroups arising from homogeneous semigroup congruences.

\begin{itemize}

\item[(i)] If $\J=\L$, then $S = \{\dot{0}\}$. 

\item[(ii)] If $\J=\{0\}$, then $S \cong \mathbb{F}_d^{*0}$.

\item[(iii)] If $\J=\C$, then $S \cong \mathbb{N}^{d,0}$.
  
\item[(iv)] If $\J$ is the ideal generated by $L_1, \ldots, L_d$, then $S \cong \{0,1\}$. 
  
\item[(v)] If $\J$ is the ideal generated by one of $L_1, ..., L_d$, then $S \cong \mathbb{F}_{d-1}^{*0}$. 

\item[(vi)] If $\J$ is the ideal generated by $L_{12}$ in $\L_2$, then $S \cong \langle e_1, e_2: e_1e_2=0\rangle$. 
\end{itemize}
\end{eg}

Let $q:\bF_d^{*0}\to S$ denote the canonical quotient map corresponding to a homogeneous semigroup congruence $\sim$. 
For $w\in \bF_d^{*0}$, we will also write $q(w)$ as $\dot{w}$. Let $[w] = q^{-1}(\dot{w})$ denote the equivalence class of $w$ modulo $\sim$.
By homogeneity, if $\dot{w} \in S_0$, then every $v \in [w]$ satisfies $|v| = |w|$.
In particular, the cardinality of the equivalence class $[w]$ is necessarily finite.
To simplify our notation, we set
\[
[s]=q^{-1}(s), \quad s\in S. 
\]

Let $\H[S]$ denote the Hilbert space with orthogonal basis $\{x_s: s\in S_0\}$ (recall that $S_0$ = $S \setminus \{\dot{0}\}$).
The inner product $(\cdot,\cdot)_S$ on $\H[S]$  is defined by 
\[
(x_s, x_t)_S=\delta_{s,t}\, \frac{1}{|[s]|}, \quad s,t \in S_0.
\]
Thus $\{y_s:= \sqrt{|[s]|}\, x_s \mid s\in S_0\}$ is an orthonormal basis of $\H[S]$.  It is convenient notationally to also set $x_{\dot 0}=0$.

We now define a weighted left regular representation of $S$ on $\H[S]$ as follows. 
For each $s\in S$, define 
\[
L_s x_t=x_{st}, \quad t \in S.
\]
Then it is easy to check that
\[
L_s y_t=\sqrt\frac{|[t]|}{|[st]|}\, y_{st}, \quad t \in S.
\]
One can check that each  $L_s$ is a contraction. Actually, we will soon see that $L_s$ 
is unitarily equivalent to $P_{\N_\J}L_w|_{\N_\J}$ for any $w \in q^{-1}(s)$. 

The \textit{semigroup algebra} $\L[S]$ is defined to be the $\wot$ closed algebra generated by the above representation of $S$. That is,
\[
\L[S]=\ol\alg^{\wot}\{L_s:s\in S\}.
\]

We are now ready to prove the main result of this subsection.

\begin{thm}
\label{T:uniequ}
Let $\J$ be a homogeneous ideal of $\L$, and let $S$ denote the corresponding semigroup as above. Then
$\L_\J$ and $\L[S]$ are unitarily equivalent.
\end{thm}

\begin{proof}
By Theorem \ref{T:Hopfideal}, we have
\[
\J=\ol{\spn}^{w^*}\{L_{u}-L_{v} \mid u,v\in \bF_d^{*0} \text{ with } u\sim v\}.
\]
Thus the closure of the range of $\J$ is
\[
\ol{\mathcal{J}\xi_\mt}=\ol{\spn}\{\xi_{u}-\xi_{v} \mid u,v\in \bF_d^{*0} \text{ with } u\sim v\}.
\]
where we use the convention that $\xi_0 = 0$. 

Recall that $\N_\J = (\J \xi_\mt)^{\perp}$. Hence if $\sum_u a_u\xi_u\in\N_\J$, then $a_u = a_v$ whenever $u \sim v$, where we use the convention that $a_0 = 0$.
Thus 
\begin{align*}
\N_\J
&=\left\{\sum_{u\in\bF_d^{*0}} a_u\xi_u \in \ell^2(\Fd): a_u=a_v\text{ if }  u\sim v \right\}\\
&=\left\{\xi\in\ell^2(\Fd): \xi=\sum_{s\in S_0} a_s \sum_{u\in [s]} \xi_u \right\},
\end{align*}
where $a_s := a_u$ for some fixed representative $u$ of $[s]$.
In particular, the set of vectors
\[
\left\{ \eta_s := \sum_{u \in [s]} \xi_u \mid s \in S_0 \right\}
\]
is an orthogonal basis for $\N_\J$.

Define an operator $U:\N_\J \to \H[S]$ by
\[
U \eta_s = |[s]|x_s, \quad s \in S_0.
\]
Since $U$ maps an orthogonal basis onto an orthogonal basis and preserves the norm, it must be a unitary.  
It remains to show that for $s \in S_0$, and any representative $w \in [s]$,
\[
L_s = U(P_{\N_\J}L_w|_{\N_\J})U^*.
\]

Fix $s \in S_0$ and $w \in [s]$ as above. On the one hand, for $t \in S_0$, we have 
\[
L_s U \eta_t = |[t]| L_s x_t = |[t]| x_{st}.
\]
On the other hand, for $t \in S_0$  and $u\in [t]$, we can write
\[
\xi_{wu} = \frac{1}{|[wu]|} \left( \sum_{v \in [wu]} \xi_v + \sum_{\substack {v \in [wu] \\ v \ne wu}} (\xi_{wu} - \xi_v) \right),
\]
where
\[
\sum_{v \in [wu]} \xi_v \in \N_\J, \quad \text{and} \quad \sum_{\substack {v \in [wu] \\ v \ne wu}} (\xi_{wu} - \xi_v) \in \N_\J^\perp.
\]
Note that $[wu] = [st]$. Using this fact and the above decomposition gives
\begin{align*}
U P_{\N_\J} L_w \eta_t &= U P_{\N_\J} L_w \sum_{u \in [t]} \xi_u \\
&= U P_{\N_\J} \sum_{u \in [t]} \xi_{wu} \\
&= U P_{\N_\J} \sum_{u \in [t]} \frac{1}{|[wu]|} \left( \sum_{v \in [wu]} \xi_v + \sum_{\substack {v \in [wu] \\ v \ne wu}} (\xi_{wu} - \xi_v) \right) \\
&= U \sum_{u \in [t]} \frac{1}{|[wu]|} \sum_{v \in [wu]} \xi_v \\
&= U \sum_{u \in [t]} \frac{1}{|[st]|} \eta_{st} \\
&= U \frac{|[t]|}{|[st]|} \eta_{st} \\
&= |[t]| x_{st}.
\end{align*}
Thus $L_s U \eta_t = U P_{\N_\J} L_w \eta_t$ for every $t \in S_0$. We conclude that $P_{\N_\J}\L|_{\N_\J}$ and $\L[S]$ are unitarily equivalent.
\end{proof}

It follows immediately from the above theorem that the $\wot$ and weak* topologies coincide on $\L[S]$,
just as they do on $\L_\J$. 

\medskip
We end this section with the following remarks.

\begin{rem}
We slightly abuse notation and write $\L[S_0]$ for the $\wot$ closed algebra generated by those $L_s$ with $s$ nonzero.
Clearly $\L[S_0]=\L[S]$, since $L_{\dot 0}=0$. The reason we use $\L[S]$ in the first place, instead of simply using $\L[S_0]$, is because $S$ is always a 
semigroup while, in general, $S_0$ is not.
For example, the noncommutative analytic Toeplitz algebra $\L$, can be identified with each of the semigroup algebras $\L[\mathbb{F}_d^{*0}]$ 
and $\L[\Fd]$.
\end{rem}

\begin{rem}
Let $\J$ be an ideal of $\L$ that is not necessarily a Hopf ideal, and is not necessarily homogeneous. From the proof of Theorem \ref{T:Hopfideal}, $\J$ determines 
a semigroup congruence $\sim$ induced by the kernel of the homomorphism $h$ defined in \eqref{E:h}. Let $S$ 
be the unital semigroup as constructed at the beginning of Subsection \ref{subS:L[S]}. 
Suppose that $\J=\ol{\spn}^{w^*}\{L_{u}-L_{v}: u\sim v\}$, and that for every nonzero $s\in S$, the equivalence class $|[s]|$ is finite.  Then 
the proof of Theorem \ref{T:uniequ} shows that $P_{\N_\J}\L|_{\N_\J}$ and $\L[S]$ are unitarily equivalent. 
\end{rem}

\section{The Hopf structure of $\L_\J$ and $(\L_\J)_*$}
\label{S:LJ}

For an arbitrary Hopf ideal $\J$ of $\L$, we show that the algebra $\L_\J$ is a Hopf dual algebra. If $\J$ is furthermore
homogeneous, then its predual $(\L_\J)_*$ is also a Hopf algebra. This generalizes some of the results in \cite{Yang12}.   

\begin{lem}
\label{L:cc}
Let $\J$ be a Hopf ideal of $\L$. Then ${(\L_\J)}_*$ is an abelian Banach algebra.  
\end{lem}

\begin{proof}
As before, we identify $\L_\J$ with $\L/\J$. Then ${(\L_\J)}_*={\L}_*\cap \J_\perp$.   
Now let $\varphi$, $\psi$ be two linear functionals in ${(\L_\J)}_*$. For $A\in \J$, it is easy to see that 
\[
\varphi\star\psi(A)=\varphi\otimes\psi ( \Delta_\L(A))=0,
\]
since $\J$ is a coideal. So $\varphi\star\psi\in(\L_\J)_*$, which implies that $(\L_\J)_*$ is a closed subalgebra of $\L_*$. 
Therefore $({(\L_\J)}_*, \star)$ is an abelian Banach algebra (\cite[Theorem 5.1]{Yang12}).
\end{proof}

\begin{thm}
\label{T:LJHopf}
Let $\J$ be an arbitrary ideal of $\L$. 
\begin{itemize}
\item[(i)]
If $\J$ is a Hopf ideal, then $\L_\J$ is a Hopf dual algebra. 
\item[(ii)] If $\J$ is homogeneous and $\L_\J$ is a Hopf dual algebra (with the canonical comultiplication), then $\J$ is a Hopf ideal. 
\end{itemize} 
\end{thm}

\begin{proof}
Let $\pi_\J: \L\to \L_\J$ denote the canonical compression map 
\[\pi_\J(A)=P_{\N_\J} A|_{\N_J}.
\]
Then $\pi_\J$ is a completely contractive and weak*-continuous epimorphism. 

(i) Notice the (algebraic) tensor 
product $\pi_\J\otimes\pi_\J: \L\otimes\L\to \L_\J\otimes\L_\J$ can be extended to a completely contractive and weak*-continuous homomorphism
from 
$\L\ol\otimes\L$ to $\L_\J\ol\otimes\L_\J$, which we continue to denote by $\pi_\J\otimes\pi_\J$.  So
\begin{align*}
\pi_\J\otimes\pi_\J: \L\ol\otimes\L&\to \L_\J\ol\otimes\L_\J,\\
 A\otimes B &\mapsto P_{\N_\J} A|_{\N_\J}\otimes P_{\N_\J} B|_{\N_J}.
\end{align*}
Now we define $\Delta_\J: \L_\J\to \L_\J\ol\otimes \L_\J$ by 
\begin{align}
\label{E:int}
\Delta_\J(\pi_\J(A))=(\pi_\J\otimes\pi_\J)\circ \Delta_\L(A)\quad \mbox{for all}\quad A\in\L.  
\end{align}
The fact that the mapping $\Delta_\J$ is well-defined follows from the relations $\ker(\pi_\J)=\J$ and  
\[
\Delta_\L(\J)\subseteq\ol{\J\otimes \L+ \L\otimes \J}^{\, w^*}\subseteq \ker(\pi_\J\otimes\pi_\J).
\] 
Note that the first inclusion above follows from the fact that $\J$ is a coideal of $\L$. 

From the definition \eqref{E:int} of $\Delta_\J$, one can see that $\Delta_\J$ is a unital weak*-continuous homomorphism. 
By \cite[Proposition 2.4.1]{Pisier}, $\Delta_\J$ is also completely contractive. 

We now verify the coassociativity of $\Delta_\J$ using the coassociativity of $\Delta_\L$. 
Using \eqref{E:int}  gives
\begin{align*}
(\id\otimes\Delta_\J)\Delta_\J(\pi_\J(A))
&=(\id\otimes\Delta_\J)[(\pi_\J\otimes\pi_\J)\circ\Delta_\L](A)\\
&=[\pi_\J\otimes(\Delta_\J\circ\pi_\J)]\circ\Delta_\L(A)\\
&=[\pi_\J\otimes((\pi_\J\otimes\pi_\J)\circ\Delta_\L)]\circ\Delta_\L(A)\\
&=(\pi_\J\otimes\pi_\J\otimes\pi_\J)\circ(\id\otimes \Delta_\L)\circ\Delta_\L(A)
\end{align*}
for all $A \in \L$, and in addition
\begin{align*}
(\Delta_\J\otimes\id)\Delta_\J(\pi_\J(A))
&=(\Delta_\J\otimes\id)[(\pi_\J\otimes\pi_\J)\circ\Delta_\L](A)\\
&=[(\Delta_\J\circ\pi_\J)\otimes\pi_\J]\circ\Delta_\L(A)\\
&=[((\pi_\J\otimes\pi_\J)\circ\Delta_\L)\otimes\pi_\J]\circ\Delta_\L(A)\\
&=(\pi_\J\otimes\pi_\J\otimes\pi_\J)\circ(\Delta_\L\otimes \id)\circ\Delta_\L(A)
\end{align*}
for all $A\in\L$. Since $\Delta_\L$ is coassociative, so is $\Delta_\J$. 

(ii) Let $\Delta$ denote the canonical comultiplication on $\L_\J$. Then $\pi_\J$ intertwines $\Delta_\L$ and $\Delta$, that is
\[
\Delta\circ \pi_\J = (\pi_\J\otimes \pi_\J)\circ \Delta_{\L}.
\] 
This implies $\Delta_\L(\J)$ is contained in $\ker(\pi_\J\otimes \pi_\J)$, since $\pi_\J(\J)=0$. 
Since $\J$ is homogeneous, $\L_\J$ has property $S_\sigma$ by Theorem \ref{T:ccap}. 
Hence $\ker(\pi_\J\otimes \pi_\J)=\ol{\J\otimes \L + \L\otimes \J}^{w^*}.$
This implies  that $\J$ is a coideal of $\L$, and hence a Hopf ideal of $\L$. 
\end{proof}

\begin{rem}
It follows from the proofs of Lemma \ref{L:cc} and Theorem \ref{T:LJHopf} that the multiplication on ${(\L_\J)}_*$, which is
inherited from $\L_*$, can be specified directly using the comultiplication on ${\L_\J}$ as
\[
\varphi\star\psi=(\varphi\otimes\psi)\circ\Delta_\J, \quad \varphi, \psi\in{(\L_\J)}_*.
\] 
\end{rem}

\begin{rem} 
Part (i) of Theorem \ref{T:LJHopf} can also be proven using the functional calculus of absolutely continuous row contractions. We briefly explain the idea here. 

Let $S_i=\pi_\J(L_i)$  for $i=1,...,d$.
We wish to show that  the map taking $S_i$ to $S_i\otimes S_i$ extends to a weak*-continuous completely contractive map on $\L_\J$.

Since the $d$-tuple $(S_1\otimes S_1,...,S_d\otimes S_d)$ can be dilated to the $d$-tuple $(L_1\otimes L_1,..., L_d\otimes L_d)$,
which is actually a (higher dimensional) unilateral shift \cite{Yang12}, it is an absolutely continuous row contraction (see for example \cite{Ken12}).
Hence there is a weak*-continuous completely contractive $\L$-functional calculus $\Phi$ satisfying $\Phi(L_i) = S_i\otimes S_i$. 
It is not hard to see that 
\[
\Phi = (\pi_\J \otimes \pi_\J)\Delta_\L.
\] 
Clearly
$\J\subseteq \ker\Phi$.
Therefore, $\Phi$ factors through a weak*-continuous completely
contractive homomorphism $\Delta_\J$ taking $S_i$ to $S_i\otimes S_i$. Therefore, 
\[
\Delta_\J\circ \pi_\J=\Phi.
\] 
The remainder of the proof follows as in the proof of Part (i) of Theorem \ref{T:LJHopf}. 
\end{rem}

Recall from \cite{Yang12} the following definition of a corepresentation of a Hopf dual algebra.
\begin{defn}
\label{D:corep}
Let $(\A, \Delta_\A)$ be a Hopf dual algebra acting on a Hilbert space $\H$.
A \textit{corepresentation} of  $(\A,\Delta_\A)$ on a Hilbert space $\K$ is an operator $V\in \A\otl\B(\K)$ satisfying
\[
V_{1,3}V_{2,3}= (\Delta_\A\otimes \id)(V).
\]
Here the notation $V_{1,3}$ and $V_{2,3}$ follows the standard leg notation. In other words, the operator $V_{1,3}$ is a linear operator on the Hilbert space
$\H\otimes\H\otimes\K$, which acts as $V$ on the first and third tensor factors, and as the identity on the second one. The operator $V_{2,3}$ is defined analogously.
\end{defn}

\begin{thm}
\label{T:conv}
Let $\J$ be a homogeneous Hopf ideal of $\L$. Then
\begin{itemize}
\item[(i)] the predual ${(\L_\J)}_*$ is a Hopf convolution algebra, and
\item[(ii)] there is a one-to-one correspondence between the corepresentations of $\L_\J$
and the completely bounded representations of ${(\L_\J)}_*$.
\end{itemize}
\end{thm}

\begin{proof}
(i) Since $\J$ is homogeneous, $\L_\J$ has property $S_\sigma$ by Theorem \ref{T:ccap}. 
It follows from Corollary \ref{C:Fubini} that we have the following completely isometric isomorphism: 
\[
(\L_\J\ol\otimes\L_\J)_*\cong{(\L_\J)}_*\widehat\otimes{(\L_\J)}_*.
\]
The result now follows as in the proof of \cite[Theorem 5.1]{Yang12}. 

(ii) The proof of this result follows as in the proof of  \cite[Theorem 6.3]{Yang12}.
\end{proof}

In contrast to Theorem \ref{T:LJHopf}, the above result is not very satisfactory since it only applies
to homogeneous Hopf ideals. The key is that the homogeneity of $\J$ ensures $\L_\J$ has property $S_\sigma$. 
This raises the following question:

\begin{ques}
\label{Q:conv}
For an arbitrary Hopf ideal $\J$ of $\L$, is $(\L_\J)_*$ a Hopf algebra? 
\end{ques}

Notice that for any Hopf ideal $\J$, $(\L_\J)_*$ is always an algebra by Lemma \ref{L:cc}.
In order to obtain a Hopf algebra structure for $(\L_\J)_*$, it is necessary to to show that it is also a coalgebra.


\section{The spectrum of ${(\L_\J)}_*$}
\label{S:spec}

Let $\J$ be a homogeneous Hopf ideal of $\L$, and let $S$ be the corresponding semigroup as in Subsection \ref{subS:L[S]}.
We will show that the spectrum of the abelian Banach algebra $(\L_\J)_*$ consists of the nonzero elements of $S$.
In particular, the spectrum of the predual  $\L_*$  of the noncommutative analytic Toeplitz algebra is $\Fd$, and the spectrum of predual $\M_*$ of the the multiplier algebra of the Drury-Arveson space is $\bN^d$. In the non-homogeneous abelian case, it is shown that if $V$ is an analytic variety of $\bB_d$ with the property that the corresponding ideal $\J_V$ is a Hopf ideal (of $\M$), then the spectrum of $(\M_V)_*$ is $\widehat{V}$, the set of nonzero semi-characters of $V$ (which necessarily form a semigroup). 

\subsection{The homogeneous case}
We will utilize the construction of the semigroup algebra $\L[S]$ from Subsection \ref{subS:L[S]} to calculate the spectrum of $(\L_\J)_*$.
The following result follows immediately from Theorem  \ref{T:uniequ}, Theorem \ref{T:LJHopf} and Theorem \ref{T:conv}. 

\begin{thm}
\label{T:hda}
The semigroup algebra $\L[S]$ is a Hopf dual algebra and the predual ${\L[S]}_*$ is a Hopf algebra. Moreover,
there is a one-to-one correspondence between the corepresentations of $\L[S]$
and the completely bounded representations of ${\L[S]}_*$.
\end{thm}

It is useful to notice that the comultiplication $\Delta_{\L[S]}$, or simply $\Delta_S$, of $\L[S]$ is determined by 
\[
\Delta_S(L_s)=L_s\otimes L_s \quad \text{for all} \quad s\in S.
\]  

\begin{rem}
One can say more in this special case:
The multiplication $m: \L[S]_*\otimes \L[S]_*\to \L[S]_*$ is surjective, and so the comultiplication $\Delta_S$ is injective. 
Indeed, suppose to the contrary that the multiplication were not surjective. Then there would be nonzero element $A\in \L[S]$ 
 such that $\langle m(\tau),A\rangle = 0$ for all $\tau \in {\L[S]}_*\otimes {\L[S]}_*$. But
 since $A\ne 0$, it has a nonzero Taylor coefficient, say the $\alpha$-th coefficient $a_\alpha\ne 0$. Let $\phi=[x_e x_\alpha^*]$.
 Then an easy computation gives $m(\phi\otimes\phi)=\phi$. However, 
 $\langle \pi(\phi\otimes \phi),A\rangle =\frac{a_\alpha}{|[\alpha]|} \ne 0$, which gives a contradiction.
\end{rem}

\begin{eg}
\label{Eg:Mk}
Fix $s\in S_0$, the set of nonzero elements of $S$. Let $\rho_s: {\L[S]}_*\to \bC$ denote the evaluation functional
\[
\rho_s(\varphi)=\varphi(L_s), \quad \varphi\in{\L[S]}_*.
\] 
Obviously, $\rho_s$ is a character of ${\L[S]}_*$. 
Moreover, it is easy to check that its corresponding corepresentation $V_{\rho_s}$ of $\L[S]$ 
in Theorem \ref{T:hda} is given by 
\[
V_{\rho_s}=L_s\otimes \id_\bC\cong L_s.
\]
\end{eg}

The following theorem greatly generalizes \cite[Theorem 6.6]{Yang12}. 

\begin{thm}
\label{T:char}
Let $\J$ be a homogeneous Hopf ideal of $\L$ with corresponding semigroup $S$. Then the spectrum of ${(\L_\J)}_*$ is $S_0$, the set of all nonzero elements of $S$. 

\end{thm}

\begin{proof}
Let $\Sigma_{\L[S]_*}$ denote the spectrum of $\L[S]_*$.  Then we want to prove that $\Sigma_{\L[S]_*}=S_0$. By Example \ref{Eg:Mk} above,  one has 
$S_0\subseteq \Sigma_{{\L[S]}_*}$. 
In order to prove the reverse inclusion, let $F\in\L[S]$ be an element of $\Sigma_{{\L[S]}_*}$. 
Then we have the following implications:
\begin{align}
\nonumber
&\langle \varphi\star\psi, F\rangle=\langle \varphi, F\rangle \langle\psi, F\rangle,
\quad \varphi, \psi \in  {{\L[S]}_*},\\
\nonumber
\Longleftrightarrow \ &
\langle (\varphi\otimes\psi), \Delta_{S}(F)\rangle=\langle \varphi\otimes \psi, F\otimes F\rangle,
\quad \varphi, \psi \in  {{\L[S]}_*},\\
\label{E:FFF} 
\Longleftrightarrow \ & \Delta_{S}(F)=F\otimes F.
\end{align}

It follows from \eqref{E:cesaro}, \eqref{E:wlim} and Theorem \ref{T:uniequ}  that every element in $\L[S]$ can be weak* approximated by the Ces\`aro sums of its Fourier series. Write the Fourier series for $F$ as
\[
F\sim\sum_{s\in S_0}a_s L_s.
\]
Then from \eqref{E:FFF} one has 
\begin{align}
\label{E:AW}
\sum_s a_s L_s\otimes \sum_s a_s L_s\ \sim\ \sum_s a_s(L_s\otimes L_s)\in\L[S]\, \ol\otimes\, \L[S].
\end{align}
Equating the $(t, t)$-th Fourier coefficients of each side of \eqref{E:AW} gives
\begin{align*}
&\left( \left(\sum_s a_s L_s\otimes \sum_s a_s L_s\right)(x_e\otimes x_e), x_t\otimes x_t\right)_S\\
&\qquad =\left(\sum_s a_s(L_s \otimes L_s)(x_e\otimes x_e), x_t\otimes x_t\right)_S,
\end{align*}
where $e$ denotes the identity of $S$. 
This implies $a_t=0$ or $1$ for every $t \in S_0$.

There must be at least one $s\in S_0$ such that $a_s\ne 0$, since $0\ne F$, and $F$ is an element of $\Sigma_{{\L[S]}_*}$. 
Suppose there are $t_1, t_2 \in S$ such that $t_1 \ne t_2$ and both $a_{t_1}$ and $a_{t_2}$ are non-zero. Then from above 
$a_{t_1}=a_{t_2}=1$.
Taking the $(t_1, t_2)$-th Fourier coefficients of each side of \eqref{E:AW}, we have
\begin{align*}
a_{t_1} a_{t_2}\frac{1}{|[t_1]||[t_2]|} &=\left( \left(\sum_s a_s L_s\otimes \sum_s a_s L_s \right)(x_e\otimes x_e), x_{t_1}\otimes x_{t_2}\right)_S\\
&=\left(\sum_s a_s(L_s\otimes L_s)(x_e\otimes x_e), x_{t_1}\otimes x_{t_2}\right)_S\\
&=\frac{a_{t_1}}{|[t_1]|^2} \delta_{t_1,t_2}\\
&=0.
\end{align*}
Obviously, this gives a contradiction. Thus $F=L_s$ for some $s\in S_0$, and hence 
$\Sigma_{{\L[S]}_*}\subseteq S_0$.  We conclude that  $\Sigma_{{\L[S]}_*} = S_0$.
\end{proof}

\begin{rem}
\label{R:groupele}
Motivated by the classical case, we will call an element $F\in \L[S]$ satisfying \eqref{E:FFF} a \textit{semigroup-like element}.
It follows from the above proof that
the set of semigroup elements of $\L[S]$ is precisely $S$. So, by Theorem \ref{T:char}, the spectrum of $\L[S]_*$ consists of all 
nonzero semigroup-like elements of $\L[S]$. 
\end{rem}

The following corollary is immediate. The spectrum of $\L_*$ is obtained in \cite[Theorem 6.6]{Yang12}. 

\begin{cor}
\label{C:char}
The spectrum of the predual $\L_*$ of the noncommutative analytic Toeplitz algebra is $\Fd$, and the spectrum of the predual $\M_*$ of the multiplier alebra of the Drury-Arveson space is  $\bN^d$.
\end{cor}


\subsection{The non-homogeneous case}
In this subsection, we consider the case of a commutative Hopf ideal $\J$  that is not necessarily homogeneous, assuming that the algebra $\L_\J$ is commutative. In this case, note that $\L_\J$ is a quotient of the multiplier algebra of the Drury-Arveson space.

Let $V\subset\bB_d$ and  $F\subset\H_d^2$. As in \cite{DavRamSha2}, we define  
\begin{align*}
\J_V&=\{f\in\M: f(\lambda)=0 \quad \forall \lambda\in V\}, \\
V(F)&=\{\lambda\in\bB_d: f(\lambda)=0 \quad \forall f\in F\}.
\end{align*}

The ball $\bB_d$ is a semigroup with the multiplication defined by component-wise multiplication: 
\[
\lambda*\mu=(\lambda_1\mu_1,...,\lambda_d\mu_d),
\]
for $\lambda=(\lambda_1,...,\lambda_d)$, $\mu=(\mu_1,...,\mu_d)$ in $\bB_d$
(see \cite{Yang12}).

\begin{lem}
\label{L:cosemi}
If $\J$ is a coideal of $\M$, the analytic variety $V(\J)$ is a sub-semigroup of $\bB_d$.
\end{lem}

\begin{proof}
Suppose $\J$ is a coideal of $\M$. Then as above, the annihilator $\J_\perp$ is a subalgebra of $\M_*$. 
Suppose $\lambda$ and $\mu$ belong to $V(\J)$.
Then $\varphi_\lambda$ and $\varphi_\mu$ belong to $\J_\perp$, and hence so does their convolution 
$\varphi_\lambda\star \varphi_\mu$, which (as one can easily check) is equal to $\varphi_{\lambda * \mu}$.
Hence $\lambda * \mu$ belongs to $V(\J)$. So $V(\J)$ is a sub-semigroup of $\bB_d$.
\end{proof}

The above result  immediately implies the following corollary.

\begin{cor}
\label{C:semi}
Let $V$ be an analytic variety of $\bB_d$ such that $\J_V$ is a coideal. Then $V$ is a semigroup
and ${(\M_V)}_*$ is an abelian Banach algebra. 
\end{cor}

\begin{proof}
The fact that $V$ is a semigroup follows from Lemma \ref{L:cosemi} and the identity $V=V(\J_V)$ is from \cite[Proposition 2.1]{DavRamSha2}.
Notice that $\J_V$ is a Hopf ideal, since $\J_V$ is an ideal. Applying Lemma \ref{L:cc}
finishes the proof. 
\end{proof}

In what follows, we describe the spectrum of ${(\M_V)}_*$ when $\J_V$ is a not necessarily homogeneous. In this case, 
the spectrum of ${(\M_V)}_*$ coincides with the set of nonzero semi-characters of $V$. 
 
Recall that $\gamma$ is a \textit{semi-character} of a semigroup $G$ if $\gamma$ is a multiplicative mapping from $G$ to the closed unit disk $\ol\bD$.
We write $\widehat{G}$ to denote the set of all nonzero semi-characters of $G$. 

\begin{prop}
\label{P:semichar}
Let $V$ be an analytic variety such that $\J_V$ is a coideal.  Then the spectrum of $(\M_V)_*$ is $\widehat{V}$.  
 \end{prop}
 
\begin{proof} 
By Corollary \ref{C:semi}, $V$ is a semigroup and ${(\M_V)}_*$ is a subalgebra of $\M_*$.  
Consider the set $G:=\{\varphi_\lambda: \lambda \in V\}$ (i.e., the set of the generators of ${(\M_V)}_*$).
Then $G$ is a semigroup with multiplication given by $\varphi_\lambda\star\varphi_\mu=\varphi_{\lambda*\mu}$.
Clearly the semigroups $G$ and $V$ are isomorphic.

Let $\gamma\in\Sigma_{{(\M_V)}_*}$. 
Then the restriction, $\gamma|_G$, of $\gamma$ to $G$ is multiplicative, 
so it gives rise to a multiplicative mapping on $V$. It suffices to show $\gamma|_G(V)\subseteq \ol\bD$.
But since $\gamma$ is an element of the spectrum of ${(\M_V)}_*$, it is necessarily contractive. Thus $|\gamma(\phi_\lambda)|\le 1$, as desired.
           
           Conversely, let $\gamma\in \widehat{V}=\widehat{G}$. Extending $\gamma$ by linearity  to ${(\M_V)}_*$ gives an element of $\Sigma_{{(\M_V)}_*}$ satisfying
\begin{align*}
\gamma\left(\sum_s a_s \varphi_s \star \sum_t b_t\varphi_t\right)
             &=\gamma\left(\sum_{s,t} a_s b_t (\varphi_s\star\varphi_t)\right) \\  
             &=\sum_{s,t} a_s b_t \gamma(\varphi_s\star\varphi_t)  \\
             &=\sum_{s,t} a_s b_t \gamma(\varphi_s)\gamma(\varphi_t) \\
             &=\gamma\left(\sum_{s} a_s\varphi_s\right)\gamma\left(\sum_{t} b_t\varphi_t\right),
\end{align*}
where the last second equality follows form the fact that $\gamma$ is multiplicative on $V$.
\end{proof} 

By the above proposition, a nonzero element $F \in \M_V$ belongs to the spectrum of ${(\M_V)}_*$ if and only if  $F(V)\subseteq\ol\bD$ and $F$ satisfies the equation 
\[
F(\lambda*\mu)=F(\lambda)F(\mu), \quad \forall \lambda, \mu \in V.       
\]

Let $V$ be an analytic variety such that $\J_V$ is a homogeneous Hopf ideal of $\M$. Then there are two 
semigroups naturally associated to $\J_V$:  the variety $V$ itself, and the semigroup $S_V$ from Subsection \ref{subS:L[S]}. 
The following corollary summarizes the relationship between them.

\begin{cor}
Let $V$ and $S_V$ be as above. Then 
 $\widehat{V}=(S_V)_0$. 
\end{cor}

\begin{proof}
This  follows immediately from Theorem \ref{T:char} and Proposition \ref{P:semichar}. 
\end{proof}

\section{Automorphisms of $\L_\J$ and ${(\L_\J)}_*$ }
\label{S:aut}

Let $\J$ be a homogeneous Hopf ideal of $\L$, and $S$ be the corresponding semigroup.
From Section \ref{S:LJ}, $\L_\J$ and $(\L_\J)_*$ are both Hopf algebras.  
In this section, we characterize the Hopf algebra automorphisms of $\L_\J$ and $(\L_\J)_*$.
As one may expect, there is a nice relationship between these two types of automorphisms.
Namely, $\theta$ is a Hopf algebra automorphism of $(\L_\J)_*$ if and only if its adjoint  
$\theta^*$ is a Hopf algebra automorphism of $\L_\J$.  
Furthermore, it turns out that these automorphisms are completely determined by semigroup automorphisms of $S$.

Throughout this section, $\J$ is assumed to be a homogeneous Hopf ideal of $\L$ with corresponding semigroup $S$.
We require the following automatic continuity result. 

\begin{lem}
\label{L:bdd}
Every endomorphism of ${\L[S]}_*$ is bounded. 
\end{lem}

\begin{proof}
Theorem \ref{T:char} clearly implies that the spectrum of $\L[S]_*$ separates the elements of $\L[S]_*$. 
Hence ${\L[S]}_*$ is a semisimple commutative Banach algebra. 
Therefore, every endomorphism on ${\L[S]}_*$ is automatically continuous (and so bounded) (see for example \cite[Lemma 11.13]{DavRamSha1}). 
\end{proof}

\begin{thm}
\label{T:aut}
A linear mapping $\theta:{\L[S]}_* \to {\L[S]}_*$ is an algebra automorphism if and only if  its adjoint $\theta^*:{\L[S]} \to {\L[S]}$ is a weak*-weak* continuous linear mapping, which  
is determined by a zero-preserving permutation $\sigma$ of $S$ in the sense that
\[
\theta^*(L_s)=L_{\sigma(s)} \quad \text{for all}\quad s\in S,
\]
where $\sigma(\dot 0)=\dot 0$. 

\end{thm}

\begin{proof}
We first claim that  
$\theta$ is an algebra automorphism of ${\L[S]}_*$ if and only if its adjoint $\theta^*$ is bijective and satisfies the following identity:
\begin{align}
\label{E:intwin}
\Delta_S\circ \theta^*=(\theta^*\otimes \theta^*)\circ \Delta_S.
\end{align}
To that end, observe that for $\varphi, \psi\in{\L[S]}_*$, we have the following implications:
\begin{align*}
&\,\theta(\varphi\star\psi)=\theta(\varphi)\star\theta(\psi),\\
\Longleftrightarrow&\, \theta(\varphi\star\psi)(A)=(\theta(\varphi)\star\theta(\psi))(A), \quad  A\in\L[S],\\
\Longleftrightarrow&\, (\varphi\star\psi)(\theta^*(A))=(\theta(\varphi)\otimes \theta(\psi))(\Delta_S(A)), \quad A\in\L[S],\\
\Longleftrightarrow&\, (\varphi\otimes\psi)\big(\Delta_S(\theta^*(A))\big)=\big((\varphi\circ\theta^*)\otimes (\psi\circ\theta^*)\big)(\Delta_S(A)), \quad A\in\L[S],\\
\Longleftrightarrow&\, (\varphi\otimes\psi)\big(\Delta_S(\theta^*(A))\big)=(\varphi\otimes \psi)\big((\theta^*\otimes\theta^*)(\Delta_S(A))\big), \quad A\in\L[S].
\end{align*}
Therefore, $\theta$ is an algebra homomorphism of ${\L[S]}_*$ if and only if $\theta^*$ satisfies \eqref{E:intwin}, which proves the claim.

Now let $\theta$ be an algebra automorphism of ${\L[S]}_*$.
By Lemma \ref{L:bdd}, $\theta$ is bounded. Using the fact that the weak* and weak topologies 
on $\L[S]$ coincide, as they do on $\L_\J$ (see Theorem \ref{T:uniequ}), it follows that $\theta^*$ is weak*-weak* continuous.
Also, it follows from \eqref{E:intwin} that 
\[
\Delta_S(\theta^*(L_s))=\theta^*(L_s)\otimes \theta^*(L_s), \quad s\in S,
\]
since $\Delta(L_s)=L_s\otimes L_s$ for all $s\in S$.
This implies that $\theta^*(L_s)$ satisfies \eqref{E:AW}, and hence that $\theta^*(L_s)$ is a semigroup-like element of $\L[S]$. Therefore,
by Remark \ref{R:groupele}, there is an element $\sigma(s) \in S$ such that
\[
\theta^*(L_s)=L_{\sigma(s)}, \quad s \in S.
\]
Clearly $\sigma(\dot 0)=\dot 0$. 

Now fix $t\in S$ and suppose  there is $A\in \L[S]$ such that $\theta^*(A)=L_t$. Then necessarily,  $A=L_s$ for some 
$s\in S$. Indeed,  without loss of generality we can assume that $t\ne 0$ and hence that $A\ne 0$. 
Write the Fourier series for $A$ as $A \sim \sum_{s \in S} a_sL_s$. 
Then 
\[
\theta^*\left(\sum_s a_sL_s\right)\sim \sum_s a_s\theta^*(L_s)\sim  \sum_sa_sL_{\sigma(s)} = L_t.
\]
From the uniqueness of the Fourier series of elements of $\L[S]$, it follows that $A=L_s$ for some $s\in S$. 
Thus we have proved that the restriction of $\theta^*$ to $\{L_s: s\in S\}$ is a bijection. 
In particular, $\theta^*$ is  determined by a permutation (i.e., bijection) of $S$, which is the map $\sigma$ defined as above.

Conversely, let $\theta^*$ be a  linear mapping on ${\L[S]}_*$ such that $\theta^*$ is weak*-weak* continuous
determined by a permutation $\sigma$ of $S$ in the above sense. By continuity, it is easy to check that $\theta^*$ satisfies \eqref{E:intwin}. 
By the claim from the beginning of the proof, we obtains that $\theta$ is an algebra homomorphism. Since $\theta^*$ is bijective, so is $\theta$. 
Therefore, $\theta$ is an algebra automorphism of ${\L[S]}_*$. 
\end{proof}

We are now able to characterize the Hopf dual algebra automorphisms of $\L_\J$ and ${(\L_\J)}_*$. 

\begin{thm}
\label{T:Hdaut}
Let $\theta:{(\L_\J)}_* \to {(\L_\J)}_*$ be a map with adjoint $\theta^*:\L_\J \to \L_\J$. Then the following statements are equivalent:
\begin{itemize}
\item[(i)] The map $\theta$ is a Hopf convolution algebra automorphism.
\item[(ii)] The adjoint map $\theta^*$ is a Hopf dual algebra automorphism.
\item[(iii)] There is a unital semigroup zero-preserving isomorphism $\sigma$ of $S$ such that 
\[
\theta^*(L_s)=L_{\sigma(s)}, \quad \forall s\in S.
\]
\end{itemize}
\end{thm}

\begin{proof}
Let $(\Delta_{S})_*$ denote the comultiplication of ${(\L_\J)}_*$. 
Then, as in the proof of Theorem \ref{T:aut},  one can show that $\theta$ intertwines $(\Delta_S)_*$
if and only if $\theta^*$ is an algebra homomorphism of $\L[S]$. 

 (i) $\Leftrightarrow$ (ii) This follows from the duality between $\L_\J$ and ${(\L_\J)}_*$ and the definitions of the comultiplications on $\L_\J$ and ${(\L_\J)}_*$.
 
(ii)  $\Leftrightarrow$ (iii) If $\theta$ is an algebra automorphism of ${\L[S]}_*$, then by Theorem \ref{T:aut}, there is a zero-preserving permutation $\sigma$ of $S$ such that 
$
\theta(L_s)=L_{\sigma(s)}
$
for all $s\in S$.  But $\theta^*$ is now also an algebra homomorphism, and one can easily check that $\sigma$ is, in addition, a homomorphism of $S$, and hence a unital semigroup isomorphism. 
\end{proof}

An immediate consequence of Theorem  \ref{T:Hdaut} is that
every Hopf dual algebra automorphism of $\L_d$ is a gauge automorphism, and hence is unitarily implemented (cf. \cite{Voi}, also \cite{DavPit}).
Using this fact, we can unify the description of the Hopf automorphisms of $\L$ and $\L_*$. Indeed, they are all unitarily implemented
by second quantizations of unitaries permuting the standard basis of $\bC^d$.

To see this, first suppose $\theta$ is a Hopf dual algebra automorphism of $\L$. Then by Theorem \ref{T:Hdaut}, $\theta$
corresponds to a permutation $\sigma$ of $\{1,...,d\}$.
Consider the Fock space $\ell^2(\Fd)$, and a fixed orthonormal basis $\{\xi_1,...,\xi_d\}$ of $\bC^d$ corresponding to $L_1,...,L_d$. In other words, 
for an element $f$ in $\bC\Fd$, $L_i f = \xi_i \otimes f$. Let $M$ denote the unitary on $\bC^d$ corresponding to $\sigma$, i.e.
$M\xi_i = \xi_{\sigma(i)}$.

The unitary $M$ induces a unitary $U$ on $\ell^2(\Fd)$ which is called the second quantization of $M$. Specifically, 
$U$ acts on elementary  tensors by
\[
U( \xi_{i_1}\otimes \cdots \otimes \xi_{i_k}) = \xi_{\sigma(i_1)} \otimes \cdots \otimes \xi_{\sigma(i_k)}, \quad k \geq 1.
\]
The map $\theta$ is implemented by the unitary $U$ in the sense that
\[
\theta(A) = UAU^*, \quad \forall A \in \L.
\]
If $\theta$ is induced by a Hopf convolution algebra automorphism of ${\L}_*$, say
$\theta_*$, then the action of $\theta_*$ is also determined by $U$ in the
following way. For an arbitrary functional in ${\L}_*$, let $F$ be a representative trace class operator. Then for all $A\in L$ we have 
\[
\langle F,\theta(A)\rangle  = \Tr(\theta(A)F) = \Tr(UAU^*F) = \Tr(AU^*FU)=\langle U^*FU, A\rangle,
\]
which implies $\theta_*([F])=[U^*FU]$.

\section{Schur Multipliers}
\label{S:mul}	

Let $\J$ be a homogeneous Hopf ideal of $\L$, and let $S$ denote the semigroup associated to it. 
In this section, we characterize the Schur multipliers of the convolution algebra $\L[S]_*$. 
We will use the notation from Section \ref{S:hhideal}. 

By \cite[Proposition 1.2]{AP1} and Theorem \ref{T:uniequ}, $\L[S]$ has property $\bA_1(1)$,   
so, for any $\varphi\in\L[S]_*$, there are $\xi, \eta\in \H[S]$ such that $\varphi=[\xi\eta^*]$. As mentioned in Section \ref{S:hhideal}, 
$\L_*$ is weakly spanned by the coefficient functionals $\varphi_w=[\xi_\mt\xi_w^*]$ ($w\in\Fd$). By the unitary equivalence between
$\L[S]$ and $\L_\J$ specified in Subsection \ref{subS:L[S]}, one immediately obtains that $\L[S]_*$ is
weakly spanned by the following functionals 
\[
\varphi_s=[x_{\dot{\mt}}x_s^*], \quad s\in S_0.
\]   

\begin{defn}
\label{D:Schur}
A function $\phi:S_0\to \bC$
is called a \textit{(completely bounded) Schur multiplier of $\L[S]_*$} if the map $m_\phi:\L[S]_*\to\L[S]_*$ defined by 
\[
m_\phi(\varphi_s)=\phi(s)\varphi_s, \quad s\in S_0,
\]
is (completely) bounded.
\end{defn}

Using an argument similar to that of \cite[Proposition 1.2]{CH}, one can easily verify that  $\phi$ is a completely bounded Schur multiplier 
of $\L[S]_*$ if and only if the adjoint  map $m_\phi^*: \L[S]\to \L[S]$ defined by
\[
 m_\phi^*(L_s) = \phi(s) L_s, \quad s \in S_0,
\]
 is a completely bounded Schur multiplier of $\L[S]$. 
 
Let $\Gamma$ be a set. Recall that a function $a:\Gamma\times \Gamma\to \bC$ is called a \textit{Schur multiplier of $B(\ell^2(\Gamma))$} if the map
$S_a: B(\ell^2(\Gamma))\to  B(\ell^2(\Gamma))$ defined by 
\[
(S_a(T)e_s, e_t)=a(s,t)(Te_s, e_t), \quad s,t \in \Gamma,
\]
is bounded, where $\{e_s:s\in \Gamma\}$ is the standard orthonormal basis of $\ell^2(\Gamma)$. 

Notice that if $A$ is a matrix of the form $A=(a(s,t))_{s,t\in\Gamma}$, then $S_a(A) = S_A(T)$, where we write $S_A(T)$ to denote the Schur (i.e. entrywise) product of $A$ and $T$. A well-known result credited to Grothendieck 
states that $a$ is a (bounded) Schur multiplier of $B(\ell^2(\Gamma))$ with $\|S_A\|\le C$
if and only if it is  a completely bounded Schur multiplier of $B(\ell^2(\Gamma))$ with $\|S_A\|_{cb}\le C$. Moreover, this occurs
if and only if there is a Hilbert space $\H$ and two functions $x,y:\Gamma\to \H$ such that $a(s,t)=(x(s),y(t))$,
with $\sup_s\|x(s)\|\sup_t\|y(t)\|\le C$. 
We refer the reader to \cite{Ari, Pisier1} for more information about Schur multipliers.

For a Hilbert space $\H$, let $S_1(\H)$ denote the algebra of trace class operators on $\H$. It is well-known that $K(\H)^*=S_1(\H)$ and $S_1(\H)^*=B(\H)$, where
$K(\H)$ and $S_1(\H)$ denote the algebras of compact operators and trace class operators on $\H$ respectively.

The following lemma seems to be folklore. We include a proof for completeness.

\begin{lem}
\label{L:mul}
The algebras $B(\H)$, $S_1(\H)$ and $K(\H)$ have the same (completely bounded) Schur multipliers. Moreover, 
every (completely bounded) Schur multiplier of $B(\H)$ and/or $S_1(\H)$ is automatically weak*-continuous. 
\end{lem}

\begin{proof}
Let $A$ be a Schur multiplier of $B(\H)$, and let $B$ and $T$ be arbitrary elements of $B(\H)$ and $S_1(\H)$, respectively. Then 
\[
\Tr(S_A(B) T)=\sum_i \sum_j (a_{ij} b_{ij}) t_{ji} = \sum_i \sum_j b_{ij} (a_{ij} t_{ji}) = \Tr(B S_{A^t}(T)),
\]
where $A^t$ is the transpose of $A$.
By the duality between $B(\H)$ and $S_1(\H)$, the above identity shows that $A$ is a Schur multiplier of $B(\H)$ iff $A^t$ is a Schur multiplier of 
$S_1(\H)$. Hence $B(\H)$ and $S_1(\H)$ have the same Schur multipliers. 

Now let $A$ be a Schur multiplier of $S_1(\H)$. Then from the above we have 
\[
\Tr(S_A(B) T)= \Tr(B S_{A^t}(T))
\]
for all $B\in S_1(\H)$ and $T\in K(\H)$.
So $A$ is a Schur multiplier of $S_1(\H)$ if and only if $A^t$ is a Schur multiplier of $K(\H)$. Hence $S_1(\H)$ and $K(\H)$ have the same Schur multipliers. 

Now suppose $S_A$ is a weak*-continuous Schur multiplier on $B(\H)$, and let $(S_A)_*$ denote its preadjoint. Then for $T\in S_1(\H)$ and $B\in B(\H)$,
\[
\langle S_A(B),T\rangle = \langle B, (S_A)_*T\rangle,
\]
so taking $T$ to be the rank one operator $T = e_ie_j^*$ gives 
\[
a_{ji} b_{ji} = \langle S_A(B), e_ie_j^*\rangle  = \langle B, (S_A)_*(e_ie_j^*)\rangle.
\]
Hence 
$(S_A)_*(e_ie_j^*)= \overline{a_{ji}} e_ie_j^*$.
In other words, the preadjoint $(S_A)_*$ is the Schur multiplier on $S_1(\H)$ corresponding to the matrix $A^*$, that is $(S_A)_*=S_{A^*}$.
Thus $S_A$ is weak*-continuous if and only if $S_{A^*}$ is a Schur multiplier of $S_1(\H)$. From above, this is the case if and only if $(A^*)^t$ is a Schur multiplier of $B(\H)$. But this is automatically true for every Schur multiplier of $B(\H)$ by Grothendieck's result.

The weak*-continuity of Schur multipliers of $S_1(\H)$ can be proved similarly. 
\end{proof}

We finish this paper by proving the main result of this section. 

\begin{prop}
Let $\phi: S_0\to \bC$ be a given function. Then the following are equivalent. 
\begin{itemize}
\item[(i)]
 $\phi$ is a completely bounded Schur multiplier of  $\L[S]_*$.
 
\item[(ii)] The adjoint $m_\phi^*$ is the restriction of a weak*-continuous Schur multiplier of $B(\H[S])$. 

 \item [(iii)] There is a Hilbert space $\K$  and two families of vectors $\{f(s) \mid s\in S_0\}$ and $\{g(s) \mid s\in S_0\}$ in $\K$ such that 
\[
\phi(s) = (f(t), g(st)), \quad s, t\in S_0 \text { with } st\ne \dot{0}.       
\]
\end{itemize}
\end{prop}

\begin{proof} 
The proof follows along the same lines as the proof of \cite[Proposition 2.3]{Ari}. From the above analysis, it suffices to show that (i) implies (ii). Suppose that $\phi: S_0\to \bC$ is a completely bounded Schur multiplier of $\L[S]_*$. To simplify notation, set $\H=\H[S]$ and $\Phi=m_\phi^*$. Then as remarked above, $\Phi:\L[S]\to \L[S]$ is a completely bounded Schur multiplier. 

By the representation theorem for completely bounded maps (see for example \cite{Pisier, Pisier1}), there is a Hilbert space $K$, a $*$-representation
$\pi: B(\H)\to B(K)$, and two bounded operators $V_i: \H\to K$ ($i=1,2$) such that 
\begin{align*}
\Phi(A)=V_2^*\pi(A)V_1, \quad A\in \L[S]. 
\end{align*}
A simple calculation shows that 
\begin{align*}
\phi(s)=|[st]|(\Phi(L_s)x_t, x_{st})_S, \quad s, t\in S_0\text{ with } st\ne \dot{0}. 
\end{align*}
Substituting the first identity above into the second one gives 
\begin{align}
\label{E:phi3}
\phi(s)=|[st]|(V_2^*\pi(L_s)V_1x_t, x_{st})_S, \quad s, t\in S_0\text{ with } st\ne \dot{0}. 
\end{align} 

Let $\{s_i\in S_0: i=1,...,k\}$ be the set of all (nonzero) images of the $d$ generators of $\Fd$. Clearly this set generates $S$. 
It follows from the results in Section \ref{S:hhideal} that each operator $\pi(L_{s_i})$ is a contraction. Let $U_i\in B(K\oplus K \oplus K)$ be a unitary dilation of $\pi(L_{s_i})$, 
so that its $(2,2)$-entry is $\pi(L_{s_i})$. For $j=1,2$, define $\widetilde{V}_j: \H\to K\oplus K \oplus K$ by 
$\widetilde{V}_j(\xi)=(0, V_j\xi, 0)$ for all $\xi\in\H$.
Then it follows from \eqref{E:phi3} that for all $s, t\in S$ with $st\ne \dot{0}$ we have
\begin{align*}
\phi(s)
&=|[st]|(V_2^*\pi(L_s)V_1x_t, x_{st})_S\\
&=|[st]|(U_s\widetilde{V}_1x_t,\widetilde{V}_2 x_{st})\\
&=(U_t^*\widetilde{V}_1x_t,|[st]| U_{st}^*\widetilde{V}_2 x_{st})\\
&=(f(t), g(st)),
\end{align*}
where $f,g: S_0\to \K:=K\oplus K \oplus K$ are given by 
\[
f(s)=U_s^*\widetilde{V}_1x_s, \quad g(s)= |[s]| U_s^*\widetilde{V}_2 x_s, \quad s\in S_0. 
\]

Let $M:B(\H)\to B(\H)$ be defined by 
\[
(MTx_t, x_s)_S=(f(t),g(s))_S(Tx_t, x_s)_S, \quad T\in B(\H). 
\]
Then $M$ is a Schur multiplier of $B(\H)$, and it is easy to check that $\Phi$ is obtained by restricting $M$.
\end{proof}


\end{document}